\documentclass[a4paper,leqno,11pt]{amsart}

\usepackage{amssymb}
\usepackage{amsthm}
\usepackage{tikz-cd}
\usepackage{mathtools}
\usepackage{framed}
\usepackage[mathscr]{eucal}
\usepackage{fullpage}
\usepackage{hyperref}
\usepackage{enumerate}
\usepackage{wasysym}
\usepackage{quiver}
\usepackage{xcolor}
\usepackage{scalerel}
\usepackage{comment}
\usepackage{faktor}
\usepackage[all]{xy}
\hypersetup{colorlinks=true,citecolor=blue,urlcolor =black,linkbordercolor={1 0 0}}
\mathtoolsset{centercolon}
\title{Infinitesimal invariants of genus \(4\) curves and the Griffiths-Pirola cycle}

\author{}
\date{\today}

\def\Z{{\mathbb Z}}

\def\C{{\mathbb C}}

\def\Q{{\mathbb Q}}
\def\proj{{\mathbb P}}

\def\cC{\mathcal{C}}

\def\cH{\mathcal{H}}
\def\cI{\mathcal{I}}

\def\cM{\mathcal{M}}

\def\cO{\mathcal{O}}

\def\cQ{\mathcal{Q}}

\def\-{\textup{-}}

\def\llra{\hbox to 10mm{\rightarrowfill}}
\def\lllra{\hbox to 15mm{\rightarrowfill}}

\def\llla{\hbox to 10mm{\leftarrowfill}}
\def\lllla{\hbox to 15mm{\leftarrowfill}}

\def\thra{\twoheadrightarrow}

\DeclareMathOperator{\Aut}{Aut}

\DeclareMathOperator{\CH}{CH}

\DeclareMathOperator{\Gal}{Gal}

\DeclareMathOperator{\Gr}{\mathsf{Gr}}

\DeclareMathOperator{\Hom}{Hom}

\DeclareMathOperator{\id}{Id}

\def\im{\mathop{\rm Im}\nolimits}

\DeclareMathOperator{\Ker}{Ker}

\DeclareMathOperator{\Kum}{Kum}

\DeclareMathOperator{\PGL}{PGL}
\DeclareMathOperator{\Pic}{Pic}

\DeclareMathOperator{\pr}{\mathsf{pr}}

\DeclareMathOperator{\Spec}{Spec}

\DeclareMathOperator{\Sym}{Sym}

\def\llra{\hbox to 10mm{\rightarrowfill}}
\def\lllra{\hbox to 15mm{\rightarrowfill}}

\newtheorem{theorem}{Theorem}[section]
\newtheorem*{theorem*}{Theorem}
\newtheorem{lemma}[theorem]{Lemma}
\newtheorem{proposition}[theorem]{Proposition}

\newtheorem{corollary}[theorem]{Corollary}
\newtheorem*{corollary*}{Corollary}

\newtheorem*{claim*}{Claim}

\theoremstyle{definition}

\theoremstyle{remark}
\newtheorem{remark}[theorem]{Remark}
\newtheorem*{remark*}{Remark}
\newtheorem*{note*}{Note}

\def\sss[#1]{{S^{[#1]}}}

\def\setminus{\smallsetminus}

\usepackage[
  backend=biber,
  style=alphabetic,
  sorting=nty,
  maxcitenames=50,
  maxnames=50
]{biblatex}

\renewbibmacro{in:}{}
\AtEveryBibitem{%
  \clearfield{issn} % Remove issn
  \clearfield{doi} % Remove doi
  \clearfield{isbn}
  \ifentrytype{online}{}{% Remove url except for @online
    \clearfield{url}
  }
}

\author{Matteo Verni}
\address{Sorbonne Université, Université Paris Cité, CNRS, IMJ-PRG, F-75005 Paris, France}
 \email{{\tt matteo.verni@imj-prg.fr}}

\begin{document}

\begin{abstract}
We give an explicit algebraic description for the infinitesimal invariants associated to the Griffiths-Pirola cycle on the universal genus \(4\) curve and its self intersection. We apply this to refine a result of Griffiths which recovers from these invariants the equations for the general genus 4 curve, and give a new proof (only in genus 4) of a nonvanishing result of Green and Griffiths.
\end{abstract}

\maketitle

%\setcounter{tocdepth}{1}
%\tableofcontents

\section{Introduction}
Griffiths' infinitesimal invariants were first introduced by Griffiths et. al. in the foundational series of papers \cite{GriffithsI}\cite{GriffithsII}\cite{GriffithsIII}. 
They were defined, for any smooth projective family \(f\colon X\to B\), by differentiating \textit{normal functions}, i.e. ``horizontal" holomorphic sections of the intermediated Jacobian \(J^{2k-1}_{X/B}\to B\) associated to the odd weight \(2k-1\) variation of polarized Hodge structures (VHS for short) on the local system \(R^{2k-1}f_* \, \underline{\Z}_X\) (\cite[Section 7.1.2]{Voisin_book_I}). 
Griffiths' theory was applied with great success to the study of algebraic cycles. 
Indeed, any codimension \(k\) cycle \(Z\in \CH^k(X)\) such that \(Z_{|X_b}\in \CH^k(X_b)_{hom}\) for all \(b\in B\) defines a normal function \[\nu_Z\colon B\to J^{2k-1}_{X/B};\]
by studying the derivative \(d \nu_Z\) obtains what is called the \textit{first infinitesimal invariant \(\delta \nu_Z\)}, which represents the first obstruction to \(Z_b \in \CH^k(X_b)_\Q\) being zero for general \(b\in B\). The original differential approach of Griffiths has continued to prove fruitful up until nowadays, as is demonstrated for example through the work of Pirola, Collino, Colombo, Frediani and others (\cite{Collino_Pirola_Duke_95}, \cite{Col_Fre_Pir_2025_Crelle}, \dots).

%Apart from the beautiful results they contain, these articles were influential as they brought to light the algebraic nature of Griffiths' invariant. 
A crucial ingredient for further developments in the study of cycle classes in families, such as \cite{Nori93} and \cite{Voisin94}, was the discovery (\cite{Voisin1988_Remarque}) that Griffith's invariant, which is constructed from differential geometry, also admits a simple cohomological interpretation. Indeed, Voisin shows that, when \(k=1\) (or more generally when the cohomology of the fibers has coniveau \(1\)),
the value of \(\nu_Z\) at the general point \(b\in B\) identifies with the restriction of the Dolbeault cohomology class \([Z]\in H^k(\Omega_X^k)\) inside \(H^k(\Omega_{X|X_b}^k)\) (see \cite[Theorem 7.14]{Voisin_book_II} for details).
This approach lends itself very well to defining invariants for higher codimensional cycles such as the one needed in \cite{Voisin94} (see \cite[page 142]{Voisin_book_II}).

%Even though the original differential approach of Griffiths has continued to prove fruitful up until nowadays, as is demonstrated for example through the work of Pirola, Collino, Colombo and others (\cite{Collino_Pirola_Duke_95}, \cite{Col_Fre_Pir_2025_Crelle}, \dots), 
%the latter approach to infinitesimal invariants based on homological algebra is sometimes advantageous. For example, it lends itself very well to defining \textit{higher order invariants}, such as the one needed in \cite{Voisin94} (see \cite[page 142]{Voisin_book_II}).
%, and the appendix of \cite{GreenGriffiths_0-cycle} for a brief discussion of their role).

\subsection{The Griffiths-Pirola normal function \texorpdfstring{\(\gamma_P\)}{ciao}}
In this article we adopt the algebraic point of view, in order to study a particularly interesting normal function \(\gamma_P\) defined away from the hyperelliptic locus, on a double cover of the moduli space of genus \(4\) curves \(\cM_4\). Its construction, which we will now illustrate, goes back to \cite{GriffithsIII}.

Let \(C\) be a non-hyperelliptic proper smooth curve of genus \(4\) over \(\C\). The canonical embedding \(C\hookrightarrow \proj(H^0(K_C))\simeq \proj^3\) is the complete intersection of a quadric \(Q_C\) and a cubic in \(\proj^3\) (see \cite[page 118]{ACGH}). 

Let \(U_0\subset \cM_4\) be the open substack parametrizing those curves \(C\) such that \(Q_C\) is smooth: for \([C]\in U_0\), \(Q_C\) is projectively equivalent to the Segre quadric surface \(\proj^1 \times \proj^1 \hookrightarrow \proj^3\).
We denote by \(L_1 ,L_2\) the two line bundles on \(C\) given by restricting the two rulings \(\cO_{\proj^1 \times \proj^1}(1,0)\) and \(\cO_{\proj^1 \times \proj^1}(0,1)\) to \(C\). We then have that \(L_1+ L_2=\cO_{\proj^3}(1)_{|C}\) and
%By setting \(L_i\coloneqq (L_i ')_{|C}\), 
\[\deg L_1 = \deg L_2=3.\]
%\deg c_1(L_i')\cap c_1(\cO_{Q_C}(3))=
%\deg c_1(\cO_{\proj^1 \times \proj^1}(1,0)\cap c_1(\cO_{\proj^1\times \proj^1} (3,3)) =
If we denote by \(J(C)\) the Jacobian of \(C\), we can define a special point of \(J(C)\) by setting \[L\coloneqq L_1- L_2\in J(C).\] 
This point is only well-defined up to a choice labelling of the two rulings: however, the image of \(L\) in the quotient \(J(C)/\pm 1\) does not depend on any choice. For this reason, if \(\cC \to \cM_4\) denotes the universal curve of genus \(4\) and \(J \to \cM_4\) the associated Jacobian fibration, then the above construction produces a section 
% https://q.uiver.app/#q=WzAsMixbMCwwLCJcXEt1bShKX1xcY0MpIl0sWzEsMCwiXFxjTV80Il0sWzAsMV0sWzEsMCwiXFxudV9QIiwyLHsiY3VydmUiOjN9XV0=
\[\begin{tikzcd}
	{\Kum(J)} & {\cM_4}
	\arrow[from=1-1, to=1-2]
	\arrow["{\gamma_P}"', curve={height=18pt}, from=1-2, to=1-1]
\end{tikzcd}\]
which vanishes precisely on \(\cM_4\setminus U_0\). Here, for any abelian scheme \(A\to B\), we denote by \(\Kum (A) \to S\) the quotient \(A/\pm 1 \to B\), whose fiber above \(b\in B\) is the \textit{Kummer variety} \[\Kum(A_b)\coloneqq A_b/\pm 1.\]
Up to a double cover \(\widetilde{\cM_4} \to \cM_4\) branched along \(\cM_4\setminus U_0\), we can resolve the ambiguity in the choice of ruling: if \(\widetilde{J}\coloneqq J\times_{\cM_4} \widetilde{\cM_4}\), we obtain a section of the base change
% https://q.uiver.app/#q=WzAsMixbMCwwLCJKX1xcY0MiXSxbMSwwLCJcXHdpZGV0aWxkZXtcXGNNXzR9Il0sWzAsMV0sWzEsMCwiXFxudV9QIiwyLHsibGFiZWxfcG9zaXRpb24iOjQwLCJjdXJ2ZSI6Mn1dXQ==
\[\begin{tikzcd}
	{\widetilde{J}} & {\widetilde{\cM_4}}
	\arrow[from=1-1, to=1-2]
	\arrow["{\gamma_P}"'{pos=0.4}, curve={height=12pt}, from=1-2, to=1-1]
\end{tikzcd}\]
 which we also denote by \(\gamma_P\). We call this the \textit{Griffiths-Pirola normal function}.

%As we implicitely used already just above, the actions of \(\langle \pm 1\rangle \) and \( G_K\) coincide on \([L]+[-L]\).

Let \(U\subset  \cM_4\) denote any open subscheme parametrizing curves \(C\subset \proj^3\) lying on smooth quadrics and let \(\widetilde{U}\coloneqq U\times_{\cM_4} \widetilde{\cM_4}\).
Denote by \(\cC\to U\) and \(J\to U\) the pullbacks along \(U \to \cM_4\) of the universal genus \(4\) curve and its relative Jacobian, respectively. Set \(\widetilde{\cC}\coloneqq \cC \times_U \widetilde{U}\) and \(\widetilde{J}\coloneqq J \times_U \widetilde{U}\).

Recall that for any \(m\in U\) we have identifications \begin{equation}\label{eq_compar_divisors}
    \Pic^0(\cC_m)=\Pic^0(\cC_m^{(2)})=Pic^0(J_m)
\end{equation}
induced by natural maps at the level of the relative Picard schemes for \(\cC/U\), \(\cC^{(2)/U}/U\) and \(J/U\), where the second family is defined to be the relative symmetric product of \( \cC \to U\).
The Griffiths-Pirola normal function \(\gamma_P \colon \widetilde{U} \to \widetilde{J}\) then corresponds to a divisor 
\begin{equation}\label{eq_def_D_Z}
    D_{\gamma_P} \in \CH^1(\widetilde{\cC}^{(2)/\widetilde{U}}) 
    %    \ \ \textup{ and } \ \  Z_{\gamma_P} \in \CH^1(\widetilde{J})
\end{equation}
whose fiber \(D_{\gamma_P} (m)\in \CH^1(\cC_m^{(2)})\)
%and \(Z_{\gamma_P} (m)\in \CH^1(J_m)\) 
is homologically trivial.
%and correspond via \eqref{eq_compar_divisors}.
While the divisor \(D_{\gamma_P}\) depends on the choice of ruling, its square does not, which means it descends to a relative 0-cycle
 \begin{equation}\label{eq_square}
    D_{\gamma_P}^2 \in \CH^2(\widetilde{\cC}^{(2)/\widetilde{U}})^{\Gal(\widetilde{K}/K)}_{\Q}\simeq \CH^2({\cC}^{(2)/{U}})_{\Q}. 
 \end{equation}
 This cycle is the main object of study of this work.
%Similarly, we also get \(Z_{\gamma_P}^2 \in \CH^2(J)_{\Q}\). 

\subsection{Main results}
The core of this article is an explicit description and analysis of the infinitesimal invariants \(\delta D_{\gamma_P}\) (Section \ref{sec_deltagamma_P}) and \(\delta D_{\gamma_P}^2\) (Sections \ref{sec_M2} and \ref{sec_restrict_to_lines}). More precisely, the infinitesimal invariant \(\delta D_{\gamma_P}(m)\) at any point \(m \in U\) belongs to the vector space \(I_1\) of \eqref{eq_where_waldone_gamma_P}, whose geometric meaning comes from its interpretation as \eqref{eq_where_is_waldeltaZ}. Likewise, the infinitesimal invariant \(\delta D_{\gamma_P}^2(m)\) belongs to the vector space \(I_2\) of \eqref{eq_who_I2}, whose geometric meaning comes from its interpretation as \eqref{eq_space_delta2}.

There is an action of \(G\coloneqq\Aut(\proj^1)\times \Aut(\proj^1)\) on the space of infinitesimal invariants \(I_{1}\) such that in the decomposition in irreducible \(G\)-subrepresentations of \(I_1\), the space \(H^0(\cO_Q(3))\) appears as a factor \(I_1^3 \subset I_1\) of multiplicity one. In particular, this space is naturally a quotient of \(I_1\). The following theorem is due to Griffiths.

%\begin{theorem}\label{thm_griffiths}(\cite[Section 6.(d)){GriffithsIII}]
%The equations cutting out a general genus \(4\) curve \(C\) are encoded in the infinitesimal variation of Hodge structure \(d_C \cP \colon H^1(T_C) \to \Hom(H^0(\Omega_C),H^1(\cO_C))\) together with the infinitesimal invariant \(\delta \gamma_P (C)\) associated to a Kuranishi family \(\cC_\Delta \to \Delta\) and the Griffiths-Pirola normal function \(\gamma_P\) restricted to \(\Delta\).
%\end{theorem}
\begin{theorem}\label{thm_griffiths}(\cite[Section 6.(d)]{GriffithsIII})
Let \(C \subset Q_C\) be a general genus \(4\) curve contained in its unique smooth quadric \(Q_C\) and defined by a cubic equation \(t\). Then \(t\) is proportional to the image in \(H^0(\cO_Q(3))\) of the infinitesimal invariant \[\delta D_{\gamma_P} (C)\in I_1.\] 
%associated to a Kuranishi family \(\cC_\Delta \to \Delta\) and the Griffiths-Pirola normal function \(\gamma_P\) restricted to \(\Delta\).
\end{theorem}

Our first main result is a refinement of the above theorem:
\begin{theorem}\label{thm_griffiths_completion}
The infinitesimal invariant \(\delta D_{\gamma_P} (C)\) belongs to the direct summand \(I_1^3\). It follows by Griffiths' theorem that the datum \(\delta D_{\gamma_P} (C)\) is equivalent to that of the cubic \(t\).
\end{theorem}

We then turn towards analyzing the square map 
\[\Sym^2 I_1 \to I_2\]
which is compatible via the infinitesimal invariant to the square map 
\[\CH^1(\cC^{(2)/U}) \ni D_{\gamma} \mapsto D_{\gamma}^2 \in \CH^2(\cC^{(2)/U}).\]
This analysis is used in \cite{Voisin_pirola} and it will play an important role in the next main theorem, which is a nonvanishing result.
\begin{theorem}\label{thm_deltagamma^2_nonvanishing}
For general \(m\in U\), the infinitesimal invariant \(\delta D_{\gamma_P}^2(m) \in I_2\) is nonzero. It follows by \cite[Proposition 1.8]{Voisin94} that
    \[D_{\gamma_P}^2  \neq 0 \ \ \textup{ in } \ \CH_0(\cC_m^{(2)})_{\Q}\]
for very general \(m \in U\).
\end{theorem}

A more precise formulation of the latter statement is the following
\begin{theorem}\label{thm_new_generator}
 Let \(\eta=\Spec{\C(\cM_4)}\) be the generic point and \(\cC_\eta\) the generic genus \(4\) curve. Then the space 
 \(\CH_0(\cC_{\eta}^{(2)})_{\Q,hom}\) is \(1\)-dimensional and generated by \(D_{\gamma_P}^2(\eta)\).
\end{theorem}

As a direct consequence of Theorem \ref{thm_new_generator}, Corollary \ref{cor_CH0_universalcurvesquared} and Remark \ref{rmk_GG_WP}, we recover the following nonvanishing result for genus \(4\) curves, which was proven for any genus as the main theorem of \cite{GreenGriffiths_0-cycle}.

\begin{corollary}
 \label{thm_GG_reproof}
Let \(\eta=\Spec{\C(\cM_4)}\) be the generic point and \(\cC_\eta\) the generic genus \(4\) curve. The ``interesting 0-cycle" of \cite{GreenGriffiths_0-cycle}, i.e. ,
\[\Gamma_{GG} \coloneqq \pr_1^* [K_{\cC_\eta}] \cap \pr_2^*[K_{\cC_\eta}] - 6 \Delta_{C *}[K_{\cC_\eta}] \]
is non-torsion in \(\CH_0(\cC_\eta\times \cC_\eta)\).    
\end{corollary}

\subsection*{Acknowledgements}
I would like to thank my PhD advisor Claire Voisin for prompting me to study infinitesimal invariants of VHS, as well as for sharing many of her ideas and insights on this problem. I would also like to thank Frédéric Han, for sharing some very useful computations obtained with the software LiE and for teaching me how to use it myself. The author benefited from the support of
ERC Synergy Grant 854361 HyperK

\section{Structure of \texorpdfstring{\(\CH_0(\cC_\eta\times \cC_\eta)_{hom}\)}{chowCKxCK}}\label{sec_structure_ofCH0(CxC)}
In this section, \(\cC \to \cM_4\) denotes the universal genus \(4\) curve over the moduli stack, and \(\eta\) the generic point of \(\cM_4\). 
The next lemma describes completely the Chow ring \(\CH^*(\cC_{\eta} \times \cC_{\eta})\). It is proven using the techniques of \cite[Section 2]{Voisin2013} (see also \cite{Fu_Laterveer_Vial2021}), which work for any family of complete intersection: we give the proof for completeness. Note that the first paragraph of the proof below amounts to the Franchetta conjecture for the special case of genus \(4\).

\begin{lemma}\label{lem_ring_of_product}
    The ring \(\CH^*(\cC_{\eta} \times \cC_{\eta})\) is generated by the \(1\)-cycles
    \begin{equation}\label{eq_generators_of_product}
        \pr_1^* ([K_{\cC_{\eta}}]) \ \ \  , \ \ \  \pr^*_2([K_{\cC_{\eta}}]) \ \ \ , \ \ \ \Delta_* [\cC_{\eta}].
    \end{equation}
\end{lemma}
\begin{proof}
Let \(B\) be the open locus in \(\proj(H^0(\cO_{\proj^3}(2)\oplus  \cO_{\proj^3} (3))\) parametrizing those pairs \((q, t)\) of nonzero quadric and cubic polynomials such that their vanishing loci \(Q\) and \(T\) are smooth surfaces with smooth intersection. Let % https://q.uiver.app/#q=WzAsMyxbMCwwLCJcXGNDIl0sWzEsMCwiXFxwcm9qXjMiXSxbMCwxLCJCIl0sWzAsMSwiXFxzaWdtYSJdLFswLDIsImYiLDJdXQ==
\[\begin{tikzcd}
	X & {\proj^3} \\
	B
	\arrow["\sigma", from=1-1, to=1-2]
	\arrow["f"', from=1-1, to=2-1]
\end{tikzcd}\] 
be the natural incidence variety, where \(X_{(Q,T)}=f^{-1}(Q,T) \xhookrightarrow{\sigma_{(Q,T)}} \proj^3 \) is the genus four curve \(Q\cap T\) with its canonical embedding. The corresponding map to the stack of genus \(4\) curves \(B \to \cM_4\) is dominant, since the general genus \(4\) curve is such a complete intersection.
For this reason, the generic fiber of \(f \colon \cC \to B\) is the generic genus \(4\) curve \(\cC_{\eta}\). We remark that \(\sigma\) is a projective bundle, hence by \cite[Theorem 3.3(b)]{Fulton_int_theory} the Chow ring of \(\cC\) is generated by the cycle classes of \(\sigma^* \cO_{\proj^3}(1)=K_{\cC/B}\) and \( f^*\cO_B(1)\). However, the latter vanishes when restricted to \(\cC_b\); since the restriction \(\CH_0(\cC)\to \CH_0(\cC_{\eta})\) is surjective by the localization exact sequence for Chow groups, we conclude that \(\CH^*(\cC_{\eta})\) is generated by only \(K_{\cC_{\eta}}\), or in other words \(\Pic(\cC_{\eta})=\Z K_{\cC_{\eta}}\). 

If we now take the incidence variety 
% https://q.uiver.app/#q=WzAsMyxbMCwwLCJcXGNDXFx0aW1lc19CIFxcY0MiXSxbMSwwLCJcXHByb2peMyBcXHRpbWVzIFxccHJval4zIl0sWzAsMSwiQiJdLFswLDEsIlxcc2lnbWEiXSxbMCwyLCJmIiwyXV0=
\[\begin{tikzcd}
	{\cC\times_B \cC} & {\proj^3 \times \proj^3} \\
	B
	\arrow["\Sigma", from=1-1, to=1-2]
	\arrow["g"', from=1-1, to=2-1]
\end{tikzcd}\]
then \(\Sigma\) is still a projective bundle away from the diagonal \(\im \Delta_{\proj^3}\) (the fiber above a distinct pair of points \((x,y)\in \proj^3\times \proj^3\) is just the projective subspace of codimension two in \(B\) of those pairs of polynomials \((q,t)\) such that \(Q\) and \(T\) pass through both \(x\) and \(y\)). 
Once again by \cite[Theorem 3.3(b)]{Fulton_int_theory}, together with the localization exact sequence for Chow groups, 
we obtain that \(CH^*(\cC \times_B \cC) \) is spanned by \((\Delta_{\cC_{\eta}})_*\CH^*(\cC_{\eta})\) together with the subring generated by \([\pr_1^*K_{\cC/B}] ,[\pr_2^*K_{\cC/B}]\) and \([g^* \cO_B(1)]\). After surjecting onto \(\CH^*(\cC_{\eta}\times \cC_{\eta})\), the latter class vanishes.
We conclude by remarking that the \(0\)-cycle \( \Delta_{\cC*} [K_{\cC}]\) can be obtained by intersecting the second and third \(1\)-cycles in \eqref{eq_generators_of_product}.

\end{proof}
For any \(X\) smooth projective variety and any two cycles \( \alpha, \beta \in \CH^*(X)\), we use the notation
\[\alpha \star \beta \coloneqq \pr_1^* \alpha \cap \pr_2^* \beta  \in \CH^*(X\times X),\]
%\[\alpha \boxtimes \beta \coloneqq \pr_1^* \alpha + \pr_2^* \beta  \in \CH^*(X\times X).\]
%If \(\dim X =1\), then 

\begin{corollary}\label{cor_CH0_universalcurvesquared}
    The abelian group \(\CH_0(\cC_{\eta} \times \cC_{\eta})_{hom}\) is generated by the cycle \[\Gamma_{GG}\coloneqq[K_{\cC_{\eta}}] \star [K_{\cC_{\eta}}] - 6 \Delta_{\cC_{\eta} *}[K_{\cC_{\eta}}].\]
\end{corollary}
\begin{remark}\label{rmk_GG_WP}
    The cycle \(\Gamma_{GG}\) is the ``interesting \(0\)-cycle" studied by Green and Griffiths in \cite{GreenGriffiths_0-cycle}. 
\end{remark}
\begin{proof}
    By Lemma \ref{lem_ring_of_product}, it suffices to take the degree zero integral combinations of pairwise intersections of the cycles in \eqref{eq_generators_of_product}. Since \((\pr_i^*[K_{\cC_{\eta}}])^2=0\) and \((\Delta_{\cC_{\eta}*}[\cC_{\eta}])^2=[\pr_i^*K_{\cC_{\eta}}] \cap \Delta_{\cC_{\eta}*}[\cC_{\eta}]\), the only generators for the abelian group \(\CH_0(\cC_{\eta} \times \cC_{\eta})_{hom}\) are
    \[ [\pr_1^* [K_{\cC_{\eta}}] \cap \pr_2^* [K_{\cC_{\eta}}] = [K_{\cC_{\eta}}] \star [K_{\cC_{\eta}}] \ \ , \ \ [\pr_i^*K_{\cC_{\eta}}] \cap \Delta_{\cC_{\eta}*}[\cC_{\eta}]= \Delta_{\cC_{\eta}*}[K_{\cC_\eta}].\]
    Moreover, their degree is computed to be
    \[\deg \big([K_{\cC_{\eta}}] \star [K_{\cC_{\eta}}] \big)=36 \ \ , \ \ \deg \Delta_{\cC_{\eta}*}[K_{\cC_{\eta}}] =6\]
    hence the claim. 
\end{proof}

%For every \(\gamma \in \CH^1(C)_0\), define 
%\[\gamma \boxtimes \gamma\coloneqq \pr_1^* \gamma + \pr_2^* \gamma \in \CH^1(C\times C).\]
%We then have \((\gamma \boxtimes \gamma)^2 =2 \, \gamma \star \gamma \in \CH_0(C\times C)_0\). 
Let \(\eta \in \cM_4\) and \(\widetilde{\eta} \in \widetilde{\cM_4}\) denote the respective generic points. 
We now show how Theorem \ref{thm_deltagamma^2_nonvanishing} implies Theorem \ref{thm_new_generator}. 
\begin{proof}[Proof of Theorem \ref{thm_new_generator}]
Just as before, for any \(Z \in \CH^1(\cC_{\widetilde{\eta}})_{hom}\), we have 
\[2 \, Z \star Z \in  \CH^2(\cC_{\eta}\times \cC_{\eta})_{hom},\] hence by Corollary \ref{cor_CH0_universalcurvesquared}, 
\(Z \star Z = \frac{h}{2} \, \Gamma_{GG}\)
for some \(h \in \Z\). In particular, if we take \(Z=\gamma_P(\widetilde{\eta})\), we get a well defined \[2(\gamma_P \star \gamma_P)(\eta) \in \CH_0(\cC_\eta \times \cC_\eta)_{hom}\] such that
\begin{equation}\label{eq_comparison_GP_GG}
    (\gamma_P \star \gamma_P)(\eta) = \frac{h}{2} \, \Gamma_{GG}. 
\end{equation} 

Let \(q \colon \cC_\eta \times \cC_\eta \to \cC_\eta^{(2)}\) denote the quotient map: then 
    \[2q_*(\gamma_P \star \gamma_P)(\eta) = D^2_{\gamma_P} (\eta) \ \ \ \textup{in} \ \ \CH^2(\cC_\eta^{(2)}).\]
By hypothesis, \(D_{\gamma_P}^2(\eta) \in \CH_0(C)_{\Q,hom}\) is nonzero, hence \(h\neq 0\). Corollary \ref{cor_CH0_universalcurvesquared} then implies that \[\CH_0(\cC_\eta^{(2)})_{hom}=\frac{1}{2h} \Z \, \cdot D_{\gamma_P}^2(\eta),\]
which concludes the proof. 

\end{proof}

\section{Infinitesimal invariants of genus \texorpdfstring{\(4\)}{} curves}\label{sec_inf_invariants}
\subsection{Infinitesimal invariants}\label{subsec_inf_invariants}
We briefly recall the construction of first infinitesimal invariants as a cohomological invariant (see \cite[Section 5]{Voisin_book_II}).

Let \(X \xrightarrow{f} B\) be a smooth projective family of varieties and \(Z\hookrightarrow X\) a codimension \(k\) cycle, to which one can associate its \textit{Dolbeault cohomology class}
\[[Z] \in H^k(\Omega_X^k)\]
Let us suppose that \(Z\) is fiberwise nullhomologous, i.e., \begin{equation}\label{eq_nullhom_condition}
    [Z]_{|X_b} \in L^1H^k(\Omega^k_{X|X_b})\coloneqq \Ker (H^k(\Omega^k_{X|X_b}) \to H^k(\Omega^k_{X_b})).
\end{equation}

This condition translates to the fact that the cycle class \([Z]\in H^{k}(X,\Omega^k_X)\) maps to zero through the composition 
\begin{equation}\label{eq_composition_1}
    H^{k}(X,\Omega^k_X) \to H^0(B,R^kf_* \Omega^k_{X})\to H^0(B,R^kf_* \Omega^k_{X/B}).
\end{equation}

The \textit{infinitesimal invariant} \(\delta Z\) is by definition the image of \([Z]\) in the middle space of the above sequence: its value at each point \(b\in B\) is just the cycle class \(\delta Z (b)=[Z]_{|X_b}\).

Following \cite[Section 5]{Voisin_book_II}), we will now describe more precisely the spaces \(L^1H^i(\Omega^k_{X|X_b})\) via the spectral sequence associated to the descending filtration \begin{equation}
    L^i\Omega_X^k\coloneqq f^*\Omega^i_B \wedge \Omega^{k-i}_X \subset \Omega^k_X
\end{equation}
induced by the cotangent exact sequence of \(f \colon X\to B\). The intermediate quotients for this filtration are of the form 
\[\Gr_L^i\Omega_X^k \simeq f^*\Omega_B^i \otimes  \Omega^{k-i}_{X/B} .\]
There is an induced filtration on higher direct images \(L^i R^kf_*\Omega^k_X \coloneqq \im ( R^kf_* L^i\Omega^k_X \to R^kf_* \Omega^k_X) \) and a spectral sequence
\begin{equation}\label{eq_ss1}
    \Omega_B^p \otimes R^{p+q}f_* \Omega_{X/B}^{k-p} =E_1^{p,q} \Longrightarrow E^{p+q}=R^{p+q}f_*\Omega^k_X .
\end{equation}
At any point \(b \in B\) we have a similar spectral sequence
\[\Omega_{B,b}^p \otimes H^{p+q}(X_b, \Omega_{X_b}^{k-p}) =E_1^{p,q} \Longrightarrow E^{p+q}=H^{p+q}(X_b,\Omega^k_{X|X_b}).\]

It is proven in \cite[Proposition 2.5]{Voisin_pirola} that the above spectral sequences degenerate at the second page. In the case of families of abelian varieties, the proof is immediate (see \cite[Remark 2.6]{Voisin_pirola}).
Recall moreover (\cite[Proposition 5.9]{Voisin_book_II}) that under the identifications \(R^{p+q}f_* \Omega_{X/B}^{k-p} \simeq \cH^{k-p,p+q}\), the differentials
\[d_1^{p,q} \colon E_1^{p,q}\rightarrow E_1^{p+1,q}\]
coincide with the \(\cO_B\)-linear maps induced by the Gauss-manin connection:
%\[\overline{\nabla}^{p,k-p} \colon \Omega_B^p \otimes R^{k}f_* \Omega_{X/B}^{k-p} \to \Omega_B^{p+1} \otimes R^{k+1}f_* \Omega_{X/B}^{k-p-1}\]
\[\overline{\nabla}^{k-p,p+q}_p \colon \Omega_B^p \otimes \cH^{k-p,p+q} \to \Omega_B^{p+1} \otimes \cH^{k-p-1,p+q+1}\]
By what has been observed above, the condition \eqref{eq_nullhom_condition} amounts to \(\delta Z \in H^0(B,L^1R^k\Omega^k_X)\), thus it defines a global section of the quotient sheaf \(\Gr_L^1R^kf_*\Omega_{X}^k=E^{1,k-1}_\infty\), i.e.,  \[\delta_1 Z \in H^0(B,E^{1,k-1}_\infty),\]
which we call the \textit{first infinitesimal invariant} of \(Z\).
By the degeneration statement mentioned above, we have
\begin{equation}\label{eq_delta_1_inclusions}
   E^{1,k-1}_\infty = E_2^{1,k-1}=\frac{\Ker \overline{\nabla}^{k-1,k}_1}{\im \overline{\nabla}^{k,k-1}_0}\subset \frac{\Omega_B \otimes \cH^{k-1,k}}{\im \overline{\nabla}^{k,k-1}_0}. 
\end{equation}

Griffiths' infinitesimal invariant \(\delta \nu_Z\) associated to the normal function \(\nu_Z\) is a global section of the latter sheaf: by \cite[Theorem 7.14]{Voisin_book_II}, we have that \(\delta_1 Z\) coincides with \(\delta\nu_Z\), i.e. ,
\[\delta \nu_Z (b)= \delta_1 Z(b) \ \ \ \textup{in} \ \ \ \frac{\Omega_{B,b}\otimes H^{k-1}(\Omega^k_{X_b})}{\overline{\nabla}^{k,k-1}_0 (H^k(\Omega^{k-1}_{X_b}))}.\]

When \(k=1\), the inclusions in \eqref{eq_delta_1_inclusions} are actually equalities, hence
\begin{equation}\label{eq_where_is_waldeltaZ}
   L^1H^1(\Omega^1_{X|X_b}) = E_2^{1,0} (b)=\frac{\Omega_{B,b} \otimes H^1(\cO_{X_b})}{ \overline{\nabla}^{1,0}_0(H^0(\Omega_{X_b}))}.  
\end{equation}
Observe also that if the relative dimension of \(X/B\) is one, then \(\delta Z= \delta_1 Z\).

\subsection{Second infinitesimal invariant}

Suppose now that \(\delta Z\in H^0(B, L^2R^kf_* \Omega^k_{X|X_b})\)
In this case \(\delta_1 Z=0\), however we can consider the image of \(\delta Z\) in the next intermediate quotient 
\begin{equation}\label{eq_space_delta2}
     \delta_2 Z \in H^0(\Gr_L^2 R^kf_*\Omega_{X}^k),
     %=H^0(E^{2,k-2}_\infty)
\end{equation}
 which we call the \textit{second infinitesimal invariant}. 

An example of such a cycle \(Z\) is given by \(Z=\alpha^2\) for an \(\alpha \in \CH^h(X)\) which is fiberwise nullhomologous as in the previous paragraph, since the Dolbeault  cycle class of an intersection product is the cup-product of the respective Dolbeault cycle classes, and the cup-product  respects the filtration, i.e., 
\[\smallsmile \colon L^iH^k(\Omega_X^k)\otimes L^j H^k(\Omega_X^k) \to L^{i+j}H^{2k}(\Omega_X^{2k}).\] 
It is clear from the definitions that \(\delta \alpha \smallsmile \delta \alpha =\delta ( \alpha^2) \in H^0(B, L^2R^kf_* \Omega^k_{X|X_b})\), which implies that the induced pairing

\begin{equation}\label{eq_graded_respected}
    \smallsmile \colon \Gr_L^iH^k(\Omega_X^k)\otimes \Gr_L^j H^k(\Omega_X^k) \to \Gr_L^{i+j}H^{2k}(\Omega_X^{2k})
\end{equation}
maps \(\delta_1 \alpha \otimes \delta_1 \alpha \) to \(\delta_2 (\alpha^2)\).

Once again by the degeneration of \eqref{eq_ss1} from \cite[Proposition 2.5]{Voisin_pirola}, we have
\begin{equation}\label{eq_space_geometric_I2}
    E^{2,k-2}_\infty = E^{2,k-2}_2= \frac{\Ker \overline{\nabla}^{k-2,k}_2}{\im \overline{\nabla}^{k-1,k-1}_1} \subseteq \frac{\Omega^2_B \otimes \cH^{k-2,k}}{\im \overline{\nabla}^{k-1,k-1}_1},
\end{equation}
and similarly at any point \(b\in B\).
\section{The first invariant  %\texorpdfstring{\(\delta \gamma_P\)}{Z=gammaPsquare}
%,  
\texorpdfstring{\(\delta D_{\gamma_P}\)}{Z=gammaP}
}\label{sec_deltagamma_P}
%Let \(\cC\to \cM_4\) be the universal curve on the moduli stack of genus \(4\) curves, let \(J\to \cM_4\) be the associated Jacobian fibration and let \(m=\{C\}\in \cM_4\) be the general point, which in particular means \(C\) is not hyperelliptic. 
%Let \(\widetilde{\cC} \to \widetilde{U}\) and \(\widetilde{J} \to \widetilde{U}\) be 
This section is devoted to the proof of Theorem \ref{thm_griffiths_completion}. We keep the same notation as in the introduction.
The Griffiths-Pirola normal function \(\gamma_P\) defines a cycle in \(\CH^1(\widetilde{\cC})\) which we still denote by \(\gamma_P\).
By \eqref{eq_where_is_waldeltaZ}, we have an associated infinitesimal invariant \(\delta \gamma_P=\delta_1 \gamma_P\) whose value at any point of the two points \(m_1, m_2 \in \widetilde{U}\) lying over \(m=[C]\) is a class 
\[\delta \gamma_P(m_i) \in \frac{H^1(\cO_{C}) \otimes H^0(K^{\otimes 2}_C)}{ \overline{\nabla}^{1,0}_0(H^0(K_{C}))}.\]

Since \(\delta \gamma_P (m_1)=- \delta \gamma_P (m_2)\), the class \(\delta \gamma_P (m)\) is well-defined up to sign. As we will now explain, the latter space does not actually depend on \(\cC \to U\) but only on the vector space \(H^0(K_C)\) equipped with the quadratic form $q\in {\rm Sym}^2V$ defining the unique quadric $Q_C$ containing $C\subset \mathbb{P}(H^0(K_C))$. We set
\[V\coloneqq H^0(K_C)\simeq H^1(\cO_C)^\ast \quad , \quad V_2\coloneqq {\rm Sym}^2V/q\simeq H^0(K_C^{\otimes2})
\simeq H^1(T_C)^* = \Omega_{\cM_4, C},\]
and recall that by \cite[Lemma 10.19, Lemma 10.22]{Voisin_book_I},
\(\overline{\nabla}^{1,0}_0 \colon V\to V^* \otimes V_2\) corresponds by adjunction to the multiplication map

\begin{equation}\label{eq_Petriboi}
    H^0(K_C)\otimes H^0(K_C)\to H^0(K_C^{\otimes 2}).
\end{equation}
In other words, \(\overline{\nabla}^{1,0}_0 \) coincides with the map induced by multiplication \[\mu \colon V\to V^* \otimes V_2\simeq \Hom(V,V_2)
%\simeq \Hom(V,\frac{\Sym^2V}{\C \, q})
\]
\[v \mapsto (w\mapsto vw)\]
We define the space of first infinitesimal invariants to be the vector space
\begin{equation}\label{eq_where_waldone_gamma_P}
   I_1 \coloneqq
    \frac{V^* \otimes V_2}{\mu(V)} .
\end{equation}
This space only depends on the pair \((V,q)\), so that any group action on \(V\) preserving \(q\) also induces one on \(I_1\). It is worth remembering that \(q\) defines a duality isomorphism \(V\simeq V^\ast\).

%We apply the previous paragraph to \(B=\widetilde{U}\), \(X=\widetilde{\cC}^{(2/\widetilde{U})}\) and \(Z= D_{\gamma_P}\) where \(\pi \colon \widetilde{\cC}^{(2/\widetilde{U})} \to \cC^{(2)/U}\). 

Recall that \(\gamma_P \in \CH^1(\cC)\) corresponds to a divisor \(D_{\gamma_P} \in \CH^1(\cC^{(2)/U})\)
%\(Z_{\gamma_P} \in \CH^1(J)\) 
by \eqref{eq_def_D_Z}. Since the two families \(\cC , \cC^{(2)/U}\) induce the same VHS of weight 1 over \(U\), their spaces of first infinitesimal invariants \eqref{eq_where_is_waldeltaZ} identify to \(I_1\), i.e.,
\begin{equation}\label{eq_L^1_all}I_1=L^1H^1(\Omega_{\cC|C})=L^1H^1(\Omega_{\cC^{(2)/U}|C})
%=L^1H^1(\Omega_{J|J_C})
\end{equation}
and the values of the invariants
%\[\delta_1 \gamma_P (m)= \delta_1 D_{\gamma_P}(m)=\delta_1 Z_{\gamma_P}(m)\in I_1.\] 
%coincides inside the space \(I_1\)
\[\delta_1 \gamma_P \quad , \quad \delta_1 D_{\gamma_P}
% \quad , \quad \delta_1 Z_{\gamma_P}
\]
at \(m \in U\), well-defined up to signs, all coincide in \(I_1\).

\subsection{\texorpdfstring{\(I_{1}\)}{si} as a representation of \texorpdfstring{\(\rm Aut(Q)\).}{boh}}\label{sec_saving_private_cubic}

%We momentarily leave aside second invariants to prove Theorem \ref{thm_griffiths_completion}.
Let us consider the group \(G\) of automorphism of the quadric \(Q\subset \proj^3\) which preserve the two rulings
\[G= \Aut(\proj^1)\times \Aut(\proj^1)\subset \Aut(Q)\subset \PGL_4(\C).\]

\begin{lemma}\label{lem_I_1_splitting}
    Recalling that the space \(I_1\) is naturally endowed with an action of \(G\), there exists a split exact sequence of \(G\)-representations 
    % https://q.uiver.app/#q=WzAsNSxbMCwwLCIwIl0sWzEsMCwiSF4wKFRfUSgxKSkiXSxbMiwwLCJJXzEiXSxbMywwLCJIXjAoXFxjT19RKDMpKSJdLFs0LDAsIjAiXSxbMCwxXSxbMSwyXSxbMiwzXSxbMyw0XSxbMywyLCJcXHNpZ21hIiwyLHsiY3VydmUiOjMsInN0eWxlIjp7ImJvZHkiOnsibmFtZSI6ImRvdHRlZCJ9fX1dXQ==
\[\begin{tikzcd}
	0 & {H^0(T_Q(1))} & {I_1} & {H^0(\cO_Q(3))} & 0
	\arrow[from=1-1, to=1-2]
	\arrow[from=1-2, to=1-3]
	\arrow[from=1-3, to=1-4]
	\arrow["\sigma"', curve={height=18pt}, dotted, from=1-4, to=1-3]
	\arrow[from=1-4, to=1-5]
\end{tikzcd}.\]
\end{lemma}
\begin{proof}
Recall that in our notation, \(\proj^3=\proj(V)\). From the Euler exact sequence 
\[0 \to \cO_{\proj^3}\to V^* \otimes \cO_{\proj^3}(1)\to T_{\proj^3} \to 0 ,\] after restricting to \(Q\) and twisting by \(\cO_Q(1)\), we get an induced long exact sequence
\[0 \to H^0(\cO_Q(1)) \to V^*\otimes H^0(\cO_Q(2)) \to H^0(T_{\proj^3|Q}(1)) \to H^1(\cO_Q(1))\simeq 0.\]
Since \(V\simeq H^0(\cO_Q(1)) \) and \(V_2=  \frac{\Sym^2V}{\C \, q}\simeq H^0(\cO_Q(2))\), the above exact sequence induces an isomorphism
\begin{equation}\label{eq_shape_I_1}
    %I_{1} \xrightarrow{\sim} H^0(T_{\proj^3|Q}(1)).
    H^0(T_{\proj^3|Q}(1)) \simeq %\frac{V \otimes V_2}{V}\simeq
    \frac{V^* \otimes V_2}{\mu(V)}=I_{1}.
\end{equation}
%where the latter equality was a definition given in \eqref{eq_where_waldone_gamma_P}.
The action of \(G\) on \(I_1\) is induced by the above identification.
By taking the long exact sequence associated to the normal sequence
\[0\to T_Q(1) \to T_{\proj^3|Q}(1) \to \cO_Q(3) \to 0\] we obtain an exact sequence
\[0\to H^0(T_Q(1)) \to H^0(T_{\proj^3|Q}(1)) \to H^0 (\cO_Q(3)) \to H^1(T_Q(1)). \]
Since \(Q\simeq \proj^1 \times \proj^1\), we have \(H^1(T_Q(1))=0\) by the Künneth formula in coherent cohomology. The arrows are all \(G\)-equivariant, so through to the identification \eqref{eq_shape_I_1} this defines the short exact sequence of \(G\)-representations we are after, \begin{equation}\label{eq_SES_I1}
    0\to H^0(T_Q(1)) \to I_1 \to H^0 (\cO_Q(3)) \to 0. \end{equation}
It turns out that \(I_1\) is the sum of three distinct irreducible \(G\)-subrepresentations we denote as X[3,3], X[1,3], X[3,1]: this can be checked either by hand or via the software LiE (\cite{LiE}), by inputting one by one the following lines of code 
\begin{center}
    setdefault A1A1 \\
    v=X[1,1] \\
    v2=X[2,2] \\
    tensor(v,v2)-v .
\end{center}
Here X[p,q] is the software's notation for the fundamental irreducible representation \(H^0(\cO_{\proj^1}(p))\otimes H^0(\cO_{\proj^1}(q))\), hence X[1,1]=\(V\) and more generally
\[\textup{X}[k,k]= H^0(\cO_Q(k))\simeq \frac{\Sym^k V}{q\cdot \Sym^{k-2} V}.\]
Since each irreducible factor appears with multiplicity one in \(I_1\), by Schur's Lemma the arrows in \eqref{eq_SES_I1} are uniquely determined up to scalar. We then consider the multiplication map
\[V\otimes V_2 \to H^0(\cO_Q(3));\]
the image of \(\mu(v) \in  V^*\otimes V_2 \simeq V\otimes V_2\) is the cubic \(q v\), which is zero in \(H^0(\cO_Q(3))\), hence the above map factors through the quotient by the subrepresentation \(\mu(V)\) into a \(G\)-equivariant
surjection
\[\pi \colon 
%\frac{V^* \otimes V_2}{\mu(V)}=
I_{1} \to H^0(\cO_Q(3)),\]
which must then be the same as the arrow in \eqref{eq_SES_I1}, up to scalar multiple. By the same representation theoretic arguments, as soon as we find a nontrivial \(G\)-equivariant map 
\[\sigma \colon H^0(\cO_Q(3)) \to I_1,\]
then this will be, up to scalar, a section of \(\pi\).
Such a map may be defined as follows. 
To any cubic \(t\in H^0(\cO_{\proj^3}(3))\) we may assign the linear map \(\sigma(t)\in \Hom(V^*,V_2)\simeq \Hom(V,V_2)\) sending the partial derivation \(\partial_i=\frac{\partial}{\partial x_i}\) to \(\partial_i t\). Since \(\partial_i(aq)\equiv a\partial_i q\) in \(V_2\), it is then clear that the subspace \(V\cdot q \subset H^0(\cO_Q(3))\) is mapped to \(\mu(V) \subset \Hom(V,V_2)\), which implies that \(t \mapsto \sigma(t)\) factors to the quotients by \(V \cdot q\) and yields the map \(\sigma\) we are after.
%\begin{equation}\label{eq_section}
  % \sigma(t) \colon \ V^* \ni \varphi \mapsto \varphi(t) \in V_2. 
%\end{equation}
We denote \(I_1^3\coloneqq \im \sigma\). This will be the subrepresentation appearing in Theorem \ref{thm_griffiths_completion}.
\end{proof}

\subsection{Recovering the cubic from \texorpdfstring{\(\delta\gamma_P\)}{massi}}Let \(C\) be a non-hyperelliptic genus \(4\) curve and let us denote by \(\cC \to B\) a Kuranishi family, so that \(C\) is the fiber above \(0 \in B\). Let \(S\) be an element of \(|\cI_C(3)|\), which parametrizes cubics passing through \(C\hookrightarrow \proj^3\). Let \(B'\subset|\cO_{S}(2)|_{sm}\) denote a small open set around the point \([C]=[S\cap Q_C]\) in the space of smooth quadric sections of \(S\). Let % https://q.uiver.app/#q=WzAsMyxbMCwwLCJcXGNRIl0sWzEsMCwiXFxwcm9qXjNcXHRpbWVzICB8XFxjT197XFxwcm9qXjN9KDIpfF97c219Il0sWzAsMSwifFxcY09fe1xccHJval4zfSgyKXxfe3NtfSJdLFswLDFdLFsxLDJdLFswLDJdXQ==
\[\begin{tikzcd}
	\cC' & {S\times  B'} \\
	{B'}
	\arrow[hook, from=1-1, to=1-2]
	\arrow[from=1-1, to=2-1]
	\arrow[from=1-2, to=2-1]
\end{tikzcd}\]
%be the associated universal quadric bundle, which is a complete family locally at \([Q_C]\). Set \(\cC'\coloneqq \cQ\cap S\times B\); 
be the universal family; \(\cC'\to B'\) is a smooth family of genus \(4\) curves, and by the universality of the Kuranishi family we have a map 
\[\phi_S \colon B' \to B\]
along which the family \(\cC \to B\) pulls back to \(\cC' \to B'\). As $|\mathcal{O}_S(2)|=|\mathcal{O}_{\proj^3}(2)|$, we also have the universal family of quadrics \begin{equation}\label{eq_quadric_bundle}
    \mathcal{Q}\subset B\times \mathbb{P}^3
\end{equation} such that $\mathcal{C}_b=\mathcal{Q}_b\cap S \subset \proj^3 $ for any $b\in B$.
The next Lemma implies that, when \(C\) and \(S\) are general, \(\phi_S\) is an isomorphism around \([C] \in B'\subset |\cO_{S}(2)|_{sm}\).

%\begin{comment}
    \begin{lemma}\label{lem_univ_family_fixed_cubic}
    %For \(S\in |\cI_C(3)|\) general, the differential at \([C]\in B'\)
    %For a general quadric \(Q \in |H^0(\cO_{\proj^3}(2))|\) and a general cubic \(S\in |H^0(\cO_{\proj^3}(3))|\), if we denote \(C=Q\cap S\), then
    For the general curve \([C] \in \cM_4\) and the general cubic \([S]\in |\cI_C(3)|\), 
    \[d\phi_S \colon H^0(\cO_{C}(2)) \to H^1(T_C)\]
    is an isomorphism.
\end{lemma}
\begin{proof}
By Serre duality, the two spaces have the same dimension, hence it suffices to prove injectivity of \(d\phi_S\). By the long exact sequence associated to the normal exact sequence
\[0 \to T_C \to T_{S|C} \to \cO_C(2) \to 0,\]
it suffices to show that \(H^0(T_{S|C})=0\).
From the normal sequence for \(S\subset \proj^3\) restricted to \(C\), we get the exact sequence
\[0\to H^0(T_{S|C}) \to H^0(T_{\proj^3|C}) \xrightarrow{\beta} H^0(\cO_C(3))\]
is exact. If \(S\) is cut out by the cubic \(s\in H^0(\cO_{\proj^3}(3))\), then the map \(\beta\) can be explicitly written as
\begin{equation}\label{eq_derive_s_1}
\left[X_i \frac{\partial}{\partial X_j}\right]\to \left(X_i \frac{\partial s}{\partial X_j}\right)_{|C}, \end{equation}
where \(\left[X_i \frac{\partial}{\partial X_j}\right]\) is the equivalence class of \(X_i \frac{\partial}{\partial X_j} \in H^0(\cO_{\proj^3}(1))\otimes V^*\) in 
\[\frac{H^0(\cO_{\proj^3}(1))\otimes V^*}{e_{\proj^3} \cdot \C}\simeq H^0(T_{\proj^3|C}).\]
Here \(e_{\proj^3} = \sum_i X_i \frac{\partial}{\partial X_i}\) denotes the Euler field, which is sent to zero by the map \eqref{eq_derive_s_1}, by the Euler identity and the fact that \(s\) vanishes on \(C\).

Let now \(J_S\subset \C[X_0, \dots , X_3]\) denote the Jacobian ideal, which by definition is generated by the partial derivatives \(\frac{\partial s}{\partial X_i}\), and consider the Jacobian ring
\[R_S\coloneqq \frac{\C[X_0, \dots , X_3]}{J_S}.\]
We denote their graded pieces of weight \(k\) by \(J^k_S\) and \(R^k_S\) respectively.
Then the map \eqref{eq_derive_s_1} factors through the isomorphism
\[H^0(T_{\proj^3|C}) \xrightarrow{\sim}  \frac{J_S^3}{s \cdot \C} \]
\[ \left[X_i \frac{\partial}{\partial X_j}\right] \mapsto X_i \frac{\partial s}{\partial X_j}\]
followed by the obvious restriction map 
\[ r_C \colon \frac{J_S^3}{s \cdot \C} \to H^0(\cO_C(3)).\] The injectivity of \(r_C\) is equivalent to the injectivity of the multiplication map 
\begin{equation}\label{eq_mult_Jacobian_ring}
    R_S^1 \xrightarrow{\cdot q} R_S^3.
\end{equation}
%Set \(q_F\coloneqq X_0 X_1 +X_2X_3\) and denote by \(O\subset|\cO_{\proj^3}(3)| \) the open set parametrizing those cubics that intersect \(Q_F\coloneqq \{q_F=0\}\) smoothly. 
Consider the variety
\[F\coloneqq \{(Q,S)\in |\cO_{\proj^3}(2)| \times |\cO_{\proj^3}(3)| ; Q\cap S \textup{ is smooth}\}.\]
%It admits an obvious map onto
%\[F\coloneqq \{(Q,S)\in \cM_4 \times |\cO_{\proj^3}(3)| ; C\subset S\},\]
%which is bijective by the considerations of the introduction, hence an isomorphism by Zariski's main theorem. 
%Let \(\tau\colon F \to |\cO_{\proj^3}(3)|\) be the natural projection,. 
The subset
\[W\coloneqq \{(Q,S)\in F ; \ R^1_S\xrightarrow{\cdot q} R^3_S \textup{ is injective}\}\]
is open in \(F\), since the map \eqref{eq_mult_Jacobian_ring} will be injective precisely when the image
\[R_S^1=\C[X_0, \dots , X_3]^1 \xrightarrow{\cdot q} \C[X_0, \dots , X_3]^3\]
is in direct sum with \(J^3_S\). It is immediate to check that the Segre quadric and the Fermat cubic define a pair which belongs to \(W\), hence \(W\) is nonempty,
%.Note that \(F\) admits an obvious map onto
%\[F_0\coloneqq \{(Q,S)\in \cM_4 \times |\cO_{\proj^3}(3)| ; C\subset S\},\]
%which is bijective by the considerations of the introduction, hence open by Zariski's main theorem. Any pair \((C,S)\) in the image of \(W\) satisfies the claim, 
which concludes the proof.

\end{proof}

%\end{comment}
We can now prove our first main theorem:

%The class \(\delta \gamma_P \in I_{1}\), which was defined only up to sign, lives in the subrepresentation \(H^0(\cO_Q(3))\).
%\end{thm}
\begin{proof}[Proof of Theorem \ref{thm_griffiths_completion}]
Let \(\cC\to B\) and \(S\) be as above. By Lemma \ref{lem_univ_family_fixed_cubic}, if \(\cQ\) is the universal quadric above \(B\) of \eqref{eq_quadric_bundle}, then up to shrinking \(B\) we have that 
\[\cC=\cQ\cap S\times B\subset \proj^3\times B.\]
The three spaces $H^1(\Omega_{\mathcal{Q}\mid Q})$, 
$H^1(\Omega_{\mathcal{Q}\mid C})$ and $H^1(\Omega_{\mathcal{C}\mid C})$ 
map by restriction to $H^1(\Omega_C)$. We denote by $H^1(\Omega_{\mathcal{Q}\mid Q})^0$ , resp.  
$H^1(\Omega_{\mathcal{Q}\mid C})^0$,  $H^1(\Omega_{\mathcal{C}\mid C})^0$  the kernels of these restrictions maps. We will denote by $\overline{H^1(T_{\mathcal{Q}\mid Q  }(-2))}, \overline{H^0(T_{\mathcal{Q}\mid C  }(1))}, \overline{H^0(T_{\mathcal{C}\mid C  }(1))}$   their respective  Serre duals.

It is easy to check that the restriction morphism  $H^1(\Omega_{\mathcal{Q}\mid Q})\rightarrow H^1(\Omega_Q)$ is an isomorphism. It follows that $H^1(\Omega_{\mathcal{Q}\mid Q})^0$ is one-dimensional, generated by the class $c_1(\mathcal{O}_{\mathcal{Q}}(1,-1))$.

The composition \[ H^1(\Omega_{\cQ|Q})^0 \to H^1(\Omega_{\cQ|C})^0 \to H^1(\Omega_{\cC|C})\]

thus has for image the space spanned by \(\delta \gamma_P \), by definition of the latter (see Section \ref{subsec_inf_invariants}). We will thus denote by $\delta_{\gamma_P}$ the composite map as well.  We consider the dual \((\delta \gamma_P )^*\) of \(\delta \gamma_P\), which fits in the following diagram

% https://q.uiver.app/#q=WzAsNyxbMiwxLCJIXjAoVF97XFxjUXxRfSgxKSkiXSxbMSwyLCJcXG92ZXJsaW5le0heMChUX3tcXG1hdGhjYWx7Q31cXG1pZCBDICB9KDEpKX0iXSxbMSwxLCJcXG92ZXJsaW5le0heMChUX3tcXG1hdGhjYWx7UX1cXG1pZCBDICB9KDEpKX0iXSxbMSwwLCJIXjAoXFxjT19DKDQpKSJdLFsxLDMsIjAiXSxbMCwxLCJcXG92ZXJsaW5le0heMShUX3tcXGNRfFF9KC0yKSl9Il0sWzMsMSwiMCJdLFswLDIsInIiXSxbMiwzXSxbMSwyXSxbNCwxXSxbMiw1XSxbMSw1LCIoXFxkZWx0YVxcZ2FtbWFfUCleKiJdLFs2LDBdLFswLDMsIlxcYWxwaGEiLDJdXQ==
\[\begin{tikzcd}
	& {H^0(\cO_C(4))} && \\
	{\overline{H^1(T_{\cQ|Q}(-2))}} & {\overline{H^0(T_{\mathcal{Q}\mid C  }(1))}} & {H^0(T_{\cQ|Q}(1))} & 0 \\
	& {\overline{H^0(T_{\mathcal{C}\mid C  }(1))}} \\
	& 0
	\arrow[from=2-2, to=1-2]
	\arrow[from=2-2, to=2-1]
	\arrow["\alpha"', from=2-3, to=1-2]
	\arrow["r", from=2-3, to=2-2]
	\arrow[from=2-4, to=2-3]
	\arrow["{(\delta\gamma_P)^*}", from=3-2, to=2-1]
	\arrow[from=3-2, to=2-2]
	\arrow[from=4-2, to=3-2]
\end{tikzcd}.\]
The column in this diagram is induced by the normal exact sequence. The row in this diagram is obtained using the fact that $\mathcal{O}_Q(-C)=\mathcal{O}_Q(-3)$.
We now observe that the universal quadric $\mathcal{Q}$ maps to $\mathbb{P}^3$ via the projection on the first factor in (\ref{eq_quadric_bundle}). We thus have a nontrivial natural map $H^0(T_{\mathcal{Q}/\proj^3}(1)) \stackrel{v}{\rightarrow }H^0(T_{\mathcal{Q}\mid Q}(1))$ 
%of vertical tangent vectors
.
We now have the following 
\begin{lemma}\label{lem_truc} The composition $\alpha\circ  v$ is identically $0$. The composite $r\circ v$ thus factors through a  map $\psi : H^0(T_{\mathcal{Q}/\proj^3}(1))\rightarrow \overline{H^0(T_{\mathcal{C}\mid C}(1))}$. 
The linear map $\psi$ is $\Aut(Q)$-equivariant.  
\end{lemma}
\begin{proof} 
First of all, we point out that \(H^0(T_{\cQ/\proj^3|Q}(1))\) has a \(\Aut(Q)\)-action defined by the fact that any element \(g\in \Aut(Q)\) extends uniquely to an automorphism \(\Aut(\cQ/B)\) (since \(Q\) is rigid) and the latter group acts on the space in question.
The existence of the map \(\psi\) follows from the fact the top vertical map in the diagram is none other than the differential \(dt \) where \(t\) is the equation of \(S\hookrightarrow \proj^3\), which vanishes on all vertical vector fields \(\eta \in H^0(T_{\cQ/\proj^3|Q}(1))\).
The fact that it is \(\Aut(Q)\)-equivariant follows essentially from the fact that the isomorphism \eqref{eq_shape_I_1} defining the \(\Aut(Q)\)-action on \(I_{1}^*=H^0(T_{\cC|C}(1))^0\) is functorial under the action of automorphisms \(g\in \Aut(Q)= \Aut(\cQ/B)\). Indeed, take such an automorphism and let \(\cC''\subset \cQ\) be the family of curves defined by \(g^*t\), the pullback of the equation cutting out \(\cC\) inside \(\cQ\) over \(B\). Then \(g\colon \cC'' \rightarrow \cC\) induces a natural map
\[g_* \colon H^0(T_{\cC|C}(1))^0 \to H^0(T_{\cC''|C''}(1))^0\]
which fits into the diagram

% https://q.uiver.app/#q=WzAsNCxbMSwxLCJIXjAoVF97XFxjUS9cXHByb2peM3xRfSgxKSkiXSxbMSwwLCJIXjAoVF97XFxjUS9cXHByb2peM3xRfSgxKSkiXSxbMCwxLCJIXjAoVF97XFxjQycnfEMnJ30oMSkpXjAiXSxbMCwwLCJIXjAoVF97XFxjQ3xDfSgxKSleMCJdLFswLDEsImdfKiJdLFsxLDMsInJcXGNpcmMgXFx1cHNpbG9uIiwyXSxbMiwzLCJnXyoiLDJdLFswLDIsInJcXGNpcmMgXFx1cHNpbG9uIiwyXV0=
\[\begin{tikzcd}
	{H^0(T_{\cC|C}(1))^0} & {H^0(T_{\cQ/\proj^3|Q}(1))} \\
	{H^0(T_{\cC''|C''}(1))^0} & {H^0(T_{\cQ/\proj^3|Q}(1))}
	\arrow["{r\circ \upsilon}"', from=1-2, to=1-1]
	\arrow["{g_*}"', from=2-1, to=1-1]
	\arrow["{g_*}", from=2-2, to=1-2]
	\arrow["{r\circ \upsilon}"', from=2-2, to=2-1]
\end{tikzcd}\]

To show \(\Aut(Q)\)-equivariance of \(\psi\), we need only to show that the left vertical map is compatible with the natural action of \(g\) on \(I_{1}^*\) via the isomorphism \eqref{eq_where_is_waldeltaZ}. But this is clear since the isomorphism is induced by a spectral sequence which is functorial with respect to maps that preserve the filtration on \(H^1(\Omega_{\cC|C})\). The dual of \(g_*\) is such a map, since it is just the pullback map on Dolbeault cohomology classes.\end{proof}
Thanks to Lemma \ref{lem_truc}
the diagram above completes as follows
% https://q.uiver.app/#q=WzAsOCxbMiwxLCJIXjAoVF97XFxjUXxRfSgxKSkiXSxbMSwyLCJcXG92ZXJsaW5le0heMChUX3tcXG1hdGhjYWx7Q31cXG1pZCBDICB9KDEpKX0iXSxbMSwxLCJcXG92ZXJsaW5le0heMChUX3tcXG1hdGhjYWx7UX1cXG1pZCBDICB9KDEpKX0iXSxbMSwwLCJIXjAoXFxjT19DKDQpKSJdLFsxLDMsIjAiXSxbMCwxLCJcXG92ZXJsaW5le0heMShUX3tcXGNRfFF9KC0yKSl9Il0sWzMsMSwiMCJdLFsyLDIsIkheMChUX3tcXGNRL1xccHJval4zfFF9KDEpKSJdLFswLDIsInIiXSxbMiwzXSxbMSwyXSxbNCwxXSxbMiw1XSxbMSw1LCIoXFxkZWx0YVxcZ2FtbWFfUCleKiJdLFs2LDBdLFswLDMsIlxcYWxwaGEiLDJdLFs3LDAsInYiXSxbNywxLCJcXHBzaSJdXQ==
\[\begin{tikzcd}
	& {H^0(\cO_C(4))} && \\
	{\overline{H^1(T_{\cQ|Q}(-2))}} & {\overline{H^0(T_{\mathcal{Q}\mid C  }(1))}} & {H^0(T_{\cQ|Q}(1))} & 0 \\
	& {\overline{H^0(T_{\mathcal{C}\mid C  }(1))}} & {H^0(T_{\cQ/\proj^3|Q}(1))} \\
	& 0
	\arrow[from=2-2, to=1-2]
	\arrow[from=2-2, to=2-1]
	\arrow["\alpha"', from=2-3, to=1-2]
	\arrow["r", from=2-3, to=2-2]
	\arrow[from=2-4, to=2-3]
	\arrow["{(\delta\gamma_P)^*}", from=3-2, to=2-1]
	\arrow[from=3-2, to=2-2]
	\arrow["v", from=3-3, to=2-3]
	\arrow["\psi", from=3-3, to=3-2]
	\arrow[from=4-2, to=3-2]
\end{tikzcd}\]

Recall from Section \ref{sec_saving_private_cubic} the notation \(G\coloneqq \Aut(\proj^1)\times \Aut(\proj^)\subset \Aut(Q)\). Next we prove 

\begin{lemma}\label{lem_G_splitting}
    We have an isomorphism of \(G\)-representations
    \[H^0(T_{\cQ/\proj^3|Q}(1))\simeq H^0(T_Q(1))\oplus V.\]
\end{lemma}
\begin{proof}
%[Proof of Lemma \ref{lem_G_splitting}]
    
Recall that we have a short exact sequence
\[0 \to H^0(T_{\cQ/\proj^3|Q}) \to H^0(\cO_Q(2))\otimes \cO_Q \to \cO_Q(2) \to 0,\]
where the right map is multiplication of sections.
It follows that we have an exact sequence
\[0\to H^0(T_{\cQ/\proj^3|Q}(1)) \to H^0(\cO_Q(2))\otimes H^0(\cO_Q(1)) \to H^0(\cO_Q(3)) \to 0.\]
The right most map is a morphism of \(G\)-representations, which by the decomposition of \(G\)-representation we computed with LiE in Lemma \ref{lem_I_1_splitting} and the fact
\[H^0(T_Q(1))\simeq H^0(\cO_Q(1,3))\oplus H^0(\cO_Q(3,1))=X[1,3] \oplus X[3,1]\]
concludes the proof of the lemma.
\end{proof}

We now conclude the proof of the theorem. The above diagram implies that \(\delta \gamma_P \in \Ker \psi^*\). Recall by Lemma \ref{lem_I_1_splitting} that 
\[ \overline{H^0(T_{\mathcal{C}\mid C  }(1))}=
I_1=I^3_1 \oplus H^0(T_Q(1))\]
as \(G\)-representations. By Lemma \ref{lem_G_splitting}, the right hand summand is the only one that appears in \(H^0(T_{\cQ/\proj^3|Q}(1))\), and it is the sum of two irreducible \(G\)-representations X[1,3] , X[3,1]. Automorphisms of \(Q\) exchanging the two rulings permute these two \(G\)-representations, which means \(H^0(T_Q(1))\) is irreducible as an \(\Aut(Q)\)-representation. Since \(\psi\) is \(\Aut(Q)\)-equivariant (by Lemma \ref{lem_truc}) and nontrivial, we deduce \(\Ker \psi^*\subset I_1\) coincides with the subrepresentation \(I_1^3\) defined at the end of the proof of the Lemma \ref{lem_I_1_splitting}, hence our claim.

\end{proof}

\section{The second infinitesimal invariant  \texorpdfstring{\(\delta D_{\gamma_P}^2\)}{Z=gammaPsquare} \label{sec_M2}
%and \texorpdfstring{\(\delta Z_{\gamma_P}^2\)}{Z=gammaPsquare}
}
%\subsection{The morphism \texorpdfstring{\(\overline{M_2}\)}{M2bar}}
%, or more indirectly by the representation-theoretic arguments of Section \ref
Similarly to what we did in Section \ref{sec_deltagamma_P} for \eqref{eq_where_is_waldeltaZ}, we will now give an explicit algebraic description of the space \(E_\infty^{2,0}(\cC^{(2)/U})\) of \eqref{eq_space_delta2}.
Consider the quadratic map

% https://q.uiver.app/#q=WzAsMixbMCwwLCJWXipcXG90aW1lcyBWXzIiXSxbMiwwLCJcXHdlZGdlXjIgVl4qXFxvdGltZXMgXFx3ZWRnZV4yVl8yIl0sWzAsMSwiTV8yIl1d
\[\begin{tikzcd}
	{V^*\otimes V_2} && {\wedge^2 V^*\otimes \wedge^2V_2}
	\arrow["{M_2}", from=1-1, to=1-3]
\end{tikzcd}\]
sending \(\psi \in \Hom(V,V_2)\) to \(\wedge^2 \psi \in \Hom(\wedge^2 V, \wedge^2 V_2)\). The notation \(M_2\) comes from the fact that if we see \(V^*\otimes V_2\) as \(\Hom(V,V_2)\) then the map is sending a matrix \(A\) to the matrix with entries the \(2\times 2\) minors of \(A\).
%The map \(M_2\) is just the wedge product, since the upper row of the diagram is the cup product in the pages of the spectral sequence, which is compatible with the cup product on the abutment, appearing in the lower row. The name \(M_2\) comes from the fact that if we see \(V^*\otimes V_2\) as \(\Hom(V,V_2)\) then the maps is sending a matrix \(A\) to the matrix with entries the \(2\times 2\) minors of \(A\). 
We can extend by linearity the quadratic map \(M_2\) above to a linear map 
\[\Sym^2 (V^*\otimes V_2 ) \xrightarrow{M_2} \wedge^2 V^*\otimes \wedge^2V_2.\]
Recalling the inclusion \(\mu(V) \hookrightarrow V^*\otimes V_2\), we are going to factor \(M_2\) through the quotient \(\Sym^2 (V^* \otimes V_2)\thra \Sym^2I_1=\Sym^2 (\frac{V^* \otimes V_2}{\mu(V)})\). Since the kernel of this quotient map is the subspace \(\mu(V)\cdot  (V^* \otimes V_2)\), defined as the image of \(\mu(V)\otimes   (V^* \otimes V_2)\) in \(\Sym^2 (V^*\otimes V_2)\), we can set
\begin{equation}\label{eq_who_I2}
I_2\coloneqq \frac{\wedge^2 V_2 \otimes \wedge^2 V^*}
{M_2\big(\mu(V) \cdot (V^* \otimes V_2)\big)} 
\end{equation}
and we obtain the diagram
% https://q.uiver.app/#q=WzAsNCxbMCwwLCJWXipcXG90aW1lcyBWXzIiXSxbMCwxLCJJX3sxLFZ9Il0sWzIsMSwiSV97MixWfSJdLFsyLDAsIlxcd2VkZ2VeMiBWXipcXG90aW1lcyBcXHdlZGdlXjJWXzIiXSxbMCwxXSxbMSwyLCJcXG92ZXJsaW5le01fMn0iXSxbMCwzLCJNXzIiXSxbMywyXV0=
\[\begin{tikzcd}
	{\Sym^2(V^*\otimes V_2)} && {\wedge^2 V^*\otimes \wedge^2V_2} \\
	{\Sym^2I_{1}} && {I_{2}}
	\arrow["{M_2}", from=1-1, to=1-3]
	\arrow[from=1-1, to=2-1]
	\arrow[from=1-3, to=2-3]
	\arrow["{\overline{M_2}}", from=2-1, to=2-3]
\end{tikzcd}.\]

Let \(C\subset Q=\{q=0\}\) be a general curve of genus \(4\), and let \[V=H^0(K_C) \ , \ V_2=H^0(K_C^{\otimes 2})=\frac{\Sym^2 V}{q}.\]
Let \(U\to \cM_4\) be an étale or analytic neighbourhood of \([C]\in \cM_4\)

\begin{proposition}\label{prop_M_2}
    The space \(I_2\) identifies with \(L^2H^2(\Omega^2_{\cC^{(2)/U}|C^{(2)}})=E^{2,0}_2 (\cC^{(2)/U})\) from \eqref{eq_space_geometric_I2} and the map \(\overline{M_2}\) identifies with the cup product in cohomology
    \begin{equation}\label{eq_equation13}
        \smallsmile \colon \Sym^2 I_1 =\Sym^2 L^1H^1(\Omega_{\cC^{(2)/U}|C^{(2)}})\to L^{2}H^{2}(\Omega_{\cC^{(2)/U}|C^{(2)}}^{2}).
    \end{equation}

\end{proposition}

%The analogous statement also holds with \(J \to U\) in place of \(\cC^{(2)/U} \to U\), although it
%The latter assertion 
%is not needed in what follows: we state it for completeness since the proof is essentially the same.
%Note  it shows us that second infinitesimal invariants of \(\cC^{(2)/U} \to U \) and \(J \to U\) coincide. 

\begin{proof}

Note that the map
\[M_2 \colon \Hom(V,V_2)^{\otimes 2} \to \Hom(\wedge^2 V , \wedge^2 V_2)\]
\[\alpha \otimes \beta \mapsto (v\wedge w\mapsto \alpha(v)\wedge \beta (w))\]
coincides up to natural identifications with the natural map
\begin{equation}\label{eq_page13truc2}(V^*\otimes V^* )\otimes (V_2 \otimes V_2)\to \wedge^2V^* \otimes \wedge^2 V_2\end{equation}

induced by the wedge product on both factors. The wedge product map \eqref{eq_page13truc2} identifies with the cup product
\[(H^1(\cO_{C^{(2)}})\otimes \Omega_{U,[C]})^{\otimes 2} \to H^2(\cO_{C^{(2)}}) \otimes \Omega_{U,[C]}^2,\]
which yields a commutative diagram
% https://q.uiver.app/#q=WzAsNCxbMCwwLCJcXFN5bV4yKFZeKlxcb3RpbWVzIFZfMikiXSxbMCwxLCJcXFN5bV4ySV97MX0iXSxbMiwxLCJcXGZyYWN7XFx3ZWRnZV4yIFZeKlxcb3RpbWVzIFxcd2VkZ2VeMlZfMn17XFxvdmVybGluZXtcXG5hYmxhfV57MSwxfV8xXFxiaWcoVl8yIFxcb3RpbWVzIEheMShcXE9tZWdhX3tDXnsoMil9fSlcXGJpZyl9Il0sWzIsMCwiXFx3ZWRnZV4yIFZeKlxcb3RpbWVzIFxcd2VkZ2VeMlZfMiJdLFswLDFdLFsxLDIsIlxcc21hbGxzbWlsZSJdLFswLDMsIk1fMiJdLFszLDJdXQ==
\begin{equation}\label{eq_truc_page13}
    \begin{tikzcd}
	{\Sym^2(V^*\otimes V_2)} && {\wedge^2 V^*\otimes \wedge^2V_2} \\
	{\Sym^2I_{1}} && {L^2H^2(\Omega^2_{\cC^{(2)/U}})}
	\arrow["{M_2}", from=1-1, to=1-3]
	\arrow[from=1-1, to=2-1]
	\arrow[from=1-3, to=2-3]
	\arrow["\smallsmile", from=2-1, to=2-3]
\end{tikzcd},
\end{equation}
where \[L^2H^2(\Omega^2_{\cC^{(2)/U}})\simeq \frac{\wedge^2 V^*\otimes \wedge^2V_2}{\overline{\nabla}^{1,1}_1\big(V_2 \otimes H^1(\Omega_{C^{(2)}})\big)}\]
by \cite[Proposition 2.5]{Voisin_pirola}.

Comparing with the diagram \eqref{eq_truc_page13}, we deduce that 
\[M_2\big(\mu(V)\cdot (V^* \otimes V_2)\big)\subset  \overline{\nabla}^{1,1}_1\big(V_2 \otimes H^1(\Omega_{C^{(2)}})\big) \]
and that the cup-product \eqref{eq_equation13} factors through map \(\overline{M_2}\). The entire proposition will then follow as soon as we can show that the above inclusion is an equality. Indeed, we can show that the maps
\[M_2 \circ (\mu \otimes \id) \colon V \otimes (V^* \otimes V_2) \to \wedge^2 V^*\otimes \wedge^2V_2\]
\[\overline{\nabla}^{1,1}_1 \colon   H^1(\Omega_{C^{(2)}}) \otimes V_2 \to \wedge^2 V^*\otimes \wedge^2V_2\]
%\[\overline{\nabla}^{1,1}_1 \colon   H^1(\Omega_{J_C}) \otimes V_2 \to \wedge^2 V^*\otimes \wedge^2V_2\]
have the same image.
We have \( V\otimes V^*\simeq H^1(\Omega_{C^{(2)}})\).
Since \(\overline{\nabla}^{1,1}_1 =\overline{\nabla}^{1,0}_0 \otimes \id_{V_2}\), under these identifications the latter map coincides with
\[ V\otimes V^*\otimes V_2 \to \wedge^2 V^*\otimes \wedge^2 V_2  \]
\[v \otimes \varphi \otimes p \mapsto \sum_i (v_i^* \wedge \varphi ) \otimes (v_i v \wedge p)\]
which is exactly \(M_2 \circ (\mu \otimes \id) \) up to the identifications mentioned in the beginning of the proof.

\end{proof}

\section{Restricting \texorpdfstring{\(\overline{M_2}(\delta D_{\gamma_P}^2)\)}{M2gamma2P} to lines in \(Q\)}\label{sec_restrict_to_lines}

In this final section we will use the explicit description of the square map given in Section \ref{sec_M2} together with Theorem \ref{thm_griffiths_completion} to prove that \(\delta D_{\gamma_P}^2(m) \neq 0\) for general \(m=[C]\in \cM_4\).

Let \(S\subseteq V\) be a \(2\)-dimensional subspace such that \(q_{|S}= 0\) and \(E\coloneqq V/S\). The restriction maps
\begin{equation} \label{eq_last1}
V^* \to S^* \ \ \textup{and } \ 
    V_2 \to \Sym^2 E,
\end{equation} where the latter is well defined since \(q_{|S}=0\), induce as morphism 
\begin{equation}\label{eq_rS}
   R_S \colon  V^* \otimes V_2 \to S^* \otimes \Sym^2 E.
\end{equation}

\begin{lemma}\label{lem_final}
    The subspace \(\mu(V)\subset V^* \otimes V_2\) vanishes under \(R_S\).
\end{lemma}
\begin{proof}
Let \(a\in V\). The map \(\mu(a) \in \Hom (V,V_2)\) ends \(x \) to \(ax\). The subspace \(S\) is the ideal of a line \(\Delta\subset Q\). For any \(x\in S\),  \(ax\) vanishes on \(\Delta\), hence in \(\Sym^2 E\).
\end{proof}
\begin{corollary}\label{lem_restricting_to_lines}
     With the same notation as above, there is a natural map
    \[R_{S,2} \colon \wedge^2 V^* \otimes \wedge^2 V_2 \to \wedge^2 S^* \otimes \wedge^2(\Sym^2E) \]
    which factors through a map
    \[r_{S,2}\colon I_{2} \to \wedge^2 S^* \otimes \wedge^2 (\Sym^2E).\]
   such that  $r_{S,2}(\overline{M_2(\phi)})=M'_2(r_{S}(\phi)$, where $M'_2$ is the 2x2 minor map for \(Hom(S^*, Sym^2E)\).
\end{corollary}
\begin{proof}
The restriction map \(R_S\) is induced by the two restriction maps \eqref{eq_last1}. Functoriality of the \(2\times2\) minor map \(M_2\) thus yields a commutative diagram

% https://q.uiver.app/#q=WzAsNCxbMCwwLCJcXFN5bV4yKFZeKiBcXG90aW1lcyBWXzIpIl0sWzEsMCwiXFx3ZWRnZV4yIFZeKiBcXG90aW1lcyBcXHdlZGdlXjIgVl8yIl0sWzAsMSwiXFxTeW1eMihTXiogXFxvdGltZXMgXFxTeW1eMkUpIl0sWzEsMSwiXFx3ZWRnZV4yIFNeKiBcXG90aW1lcyBcXHdlZGdlXjIgXFxTeW1eMkUiXSxbMCwxLCJNXzIiXSxbMCwyLCJcXFN5bV4yUl9TIiwyXSxbMiwzLCJNX3syfSciXSxbMSwzLCJSX3tTLDJ9Il1d
\[\begin{tikzcd}
	{\Sym^2(V^* \otimes V_2)} & {\wedge^2 V^* \otimes \wedge^2 V_2} \\
	{\Sym^2(S^* \otimes \Sym^2E)} & {\wedge^2 S^* \otimes \wedge^2 \Sym^2E}
	\arrow["{M_2}", from=1-1, to=1-2]
	\arrow["{\Sym^2R_S}"', from=1-1, to=2-1]
	\arrow["{R_{S,2}}", from=1-2, to=2-2]
	\arrow["{M_{2}'}", from=2-1, to=2-2]
\end{tikzcd}\]
%\[\wedge^2 V^* \otimes \wedge^2 V_2 \to \wedge^2 S^* \otimes \wedge^2 (\Sym^2E )\]

By Lemma \ref{lem_final}, the left vertical map vanishes on \(\mu(V)\cdot (V^* \otimes V_2)\), hence the right vertical map factors through \(I_2\). This gives the desired map \(r_{S,2}\).

\end{proof}

We finally prove
\begin{proposition}\label{prop_nonvanishing}
For the general \(m\in \cM_4\), there exists a subspace \(S\subseteq V\) as in Lemma \ref{lem_restricting_to_lines} such that we have
\[r_{S,2}(\delta D_{\gamma_P}^2(m))\neq 0.\]
\end{proposition}

\begin{corollary*}[Theorem \ref{thm_deltagamma^2_nonvanishing}]
    The infinitesimal invariant \(\delta D_{\gamma_P}^2(m) \in I_2\) is nonzero. In particular,
    \[D_{\gamma_P}^2  \neq 0 \ \ \ \textup{in}\  \  \CH^2(\cC_m^{(2)})_{\Q}.\]
\end{corollary*}
\begin{proof}[Proof of Proposition \ref{prop_nonvanishing}]
    By Section \ref{sec_deltagamma_P} 
    %\eqref{eq_where_waldone_gamma_P}
    and Theorem \ref{thm_griffiths_completion}, \(\delta D_{\gamma_P}(m)\) defines up to sign a class in \(I_{1}=\frac{\Hom(V,V_2)}{\mu(V)}\) which is represented by \(\sigma(t_m)\in \Hom(V,V_2)\) where \(t_m\in H^0(\cO_Q(3))\) is the unique cubic on \(Q\) which cuts out \(C\)
    and \(\sigma\colon H^0(\cO_Q(3))\to V^*\otimes V_2\) is the map defined in 
    Lemma \ref{lem_I_1_splitting}, sending a cubic polynomial \(t\) to \(\sum_i v_i^* \otimes v_i^*(t)\).
    We have 
    \begin{equation}\label{eq_look_at_this}
        r_{S,2}(M_2(\sigma(t)^2))=v^*_1\wedge v_2^* \otimes v^*_1(t)\wedge v_2^*(t).
    \end{equation}
%To conclude, it suffices to find a line \(\proj(S)\subset Q\) such that \(S\) is spanned by two vectors of the form \(\partial_1 (q), \partial_2(q)\) where \(\partial_1,\partial_2\in V^*\) such that \(\partial_1(t),\partial_2(t)\in \Sym^2 E\) are linearly indipendent. To do this, 
We need to show that the locus \(W\subset \cM_4\) of those non-hyperelliptic points \(m\in \cM_4\) such that there exists an \(S\) for which the above expression is nonzero at \(t=t_m\) is open and nonempty. 
To prove nonemptiness, we take the curve cut out by the Segre quadric \(q_0\) and the cubic 
\[t_\epsilon \coloneqq (X_0 + \epsilon X_2)^3+(X_1 + \epsilon X_3)^3 +X_2^3 + X^3_3\]
for a small choice of \(\epsilon \in \C^*\).
Indeed, for \(\epsilon =0\) the curve is smooth so it will still be smooth for a small nonzero \(\epsilon.\) Moreover, we can take \(S_0\coloneqq \langle X_0,X_1\rangle_\C\) and have
\[r_{S_0,2}(M_2(\sigma(t_0)^2))
%\frac{\partial}{\partial X_0}\wedge \frac{\partial}{\partial X_1} \otimes 3(X_0+\epsilon X_2)^2\wedge 3(X_1+\epsilon X_3)^2
\equiv \frac{\partial}{\partial X_0}\wedge \frac{\partial}{\partial X_1} \otimes 3 \epsilon^2 X^2_2\wedge 3\epsilon^2 X^2_3 \neq 0 .\]

Let $\cM_4'$ be the space  of smooth complete intersections $C$ of a quadric $Q$ and a cubic  in \(\mathbb{P}^3\), which admits a natural map \(\pi \colon \cM_4' \to \cM_4\) by the universal property of the latter.
Let  \(Y\coloneqq \{([C],S) \in \cM_4'\times \Gr(2,V)| \proj(S)\subset Q_C\}\) and consider the subspace \(W'\subset Y\) parameterizing  pairs $([C],S)$, with 
\(r_{S,2}(M_2(\sigma(t_{[C]})^2))\neq0\). If $\pr_1: Y\rightarrow \cM_4'$ denotes the projection on the first factor, we then have
\[\pr_1(W')=\pi^{-1}(W).\]
Since \(\pr_1\) is open
and \(\pi(\pi^{-1}(W))= W\) (as \(\pi\) is surjective on the non-hyperelliptic locus),
 it suffices to show that \(W'\) is open.

Recall the notation $R_S:I_{1}\rightarrow \Hom(S,\Sym^2E)$ of (\ref{eq_rS}).  As \(([C],S)\) varies in \(Y\), $R_S(\delta D_{\gamma_P})$ defines a morphism $$R(\delta D_{\gamma_P}) \colon \mathcal{S}\rightarrow Sym^2\mathcal{E} $$
between the tautological vector bundles  on $Gr(2,V)$, pulled-back to  $Y$. The subset $W'$ is by Corollary 6.2 the locus where  $R(\delta D_{\gamma_P})$ has maximal rank, hence it is open, which concludes the proof.

\begin{comment}
    
By Lemma \ref{lem_univ_family_fixed_cubic}, all small deformations of the curve \(C=V(t) \cap Q\) corresponding to a point \(m_0 \in U\) are intersections of the same cubic \(V(t)\) and small deformations \(\cQ_m\) of the quadric \(Q\). For any chosen small deformation of the quadric \(\cQ \to \Delta\), we can deform the line \(\ell_0\coloneqq \proj(S_0)\) along with \(Q\) inside \(\cQ\), so that we have a family of lines \(\ell \subset \proj^3 \times \Delta\) over \(\Delta\) with \(\ell_m \subset \cQ_m,\) which corresponds to a family of \(2\)-planes \(S_m\subset V\).

Looking at \eqref{eq_look_at_this}, it is clear that \(r_{S,2}(M_2(\sigma(t)^2))\) is algebraic in both the coefficients of the linear polynomials cutting out \(S\) and in the coefficients of \(t\), so if
\[m \mapsto r_{S_m,2}(M_2(\sigma(t)^2))\] were nonzero for \(m=m_0\), it would stay nonzero for any \(m \in \Delta\) and \(\Delta\) a small enough euclidean neighbourhood. This proves the analytic openness of the locus we are interested in, and the proof is complete.

\end{comment}

\end{proof}

\printbibliography
%\bibliography{main_last}
\end{document}